\documentclass[a4paper]{amsart}
\usepackage{amsmath, amssymb, amsthm, eucal}
\usepackage[left=1.2in, right=1.2in, top=1.5in, bottom=1.5in]{geometry}
\usepackage{graphicx}
\usepackage[shortlabels]{enumitem}

\usepackage[below]{placeins} 
\usepackage{subcaption}
\usepackage{sidecap}
\usepackage{wrapfig}
\usepackage{float}
\usepackage{xcolor}

\usepackage{mismath}

\usepackage[capitalise]{cleveref}
\newtheorem{theorem}{Theorem}
\newtheorem{proposition}[theorem]{Proposition}
\newtheorem{lemma}[theorem]{Lemma}
\newtheorem{cor}[theorem]{Corollary}
\newtheorem{definition}[theorem]{Definition}

\numberwithin{equation}{section}
\numberwithin{theorem}{section}

\crefname{lemma}{Lemma}{Lemmas}
\crefname{cor}{Corollary}{Corollaries}
\crefname{proposition}{Proposition}{Propositions}
\crefname{definition}{Definition}{Definitions}
\usepackage[maxbibnames=5]{biblatex} 
\usepackage{multirow}
\usepackage[foot]{amsaddr}
\def\l{\left(}
\def\r{\right)}
\def\LGU{LGI}

\def\archi{\textbf{(d,k,s)}}
\def\Loss{\ensuremath{\mathcal{L}_{X,Y}^{\archi}}}

\def\Losskurz{\ensuremath{\mathcal{L}}} 
\def\LosskurzN{\ensuremath{\Losskurz\l \wall\r} }

\def\wNs{\ensuremath{w^{(1)},\dots,w^{(N)}} }
\def\wNsv#1{\ensuremath{w^{(1)}(#1),\dots,w^{(N)}(#1)} }

\def\wi{\ensuremath{w^{(i)}} }

\def\Wi{\ensuremath{W^{(i)}} }
\def\WI#1{\ensuremath{W^{(#1)}} }

\def\wI#1{\ensuremath{w^{(#1)}}} 

\def\w{\ensuremath{v} }
\def\v{\ensuremath{\w} }
\def\W{\ensuremath{V} }
\def\V{\ensuremath{\W} }
\def\k{\ensuremath{k_{v}} }
\def\s{\ensuremath{s_{v}}}

\newcommand{\triple}[1]{|\!|\!| #1 |\!|\!|}
\newcommand{\bigtriple}[1]{\Big|\!\Big|\!\Big| #1 \Big|\!\Big|\!\Big|}

\def\ksum{\ensuremath{\bar{k}}}
\def\wall{\ensuremath{\vec{w}} }
\def\wallt{\ensuremath{\vec{w}(t)}}

\def\dif#1{\frac{d}{d #1}}
\def\Tildepi{\Tilde{p}_{i}}

\def\R{\mathbb{R}}
\def\N{\mathbb{N}}

\def\C{\mathbb{C}}

\newcommand{\ul}[1]{[#1]}

\begin{document}

\title[Convergence of gradient flow for learning convolutional neural networks]{Convergence of gradient flow for learning convolutional neural networks}
\author{Jona-Maria Diederen$^1$}
\address{$^1$aiXbrain GmbH, Aachen, Germany}
\email{Jona.diederen@aixbrain.de}
\author{Holger Rauhut$^{2,3}$}
\author{Ulrich Terstiege$^{2,3}$}
\address{$^2$Department of Mathematics, Ludwig-Maximilians-Universität, Theresienstr.\ 39, 80333 München, Germany}
\address{$^3$Munich Center for Machine Learning, Ludwig-Maximilians-Universität München, Germany}
\email{rauhut@math.lmu.de}
\email{terstiege@math.lmu.de}
\date{\today}

\begin{abstract}
%% Text of abstract
Convolutional neural networks are widely used in imaging and image recognition. Learning such networks from training data leads to the minimization of a non-convex function. This makes the analysis of standard optimization methods such as variants of (stochastic) gradient descent challenging. In this article we study the simplified setting of linear convolutional networks. We show that the gradient flow (to be interpreted as an abstraction of gradient descent) applied to the empirical risk defined via certain loss functions including the square loss always converges to a critical point, under a mild condition on the training data.
\end{abstract}

\maketitle
\section{Introduction}
Convolutional neural networks represent a powerful machine learning architecture that has proven very successful for various machine vision applications such as image recognition or medical imaging. For adapting a convolutional neural network to a particular task based on training data, one usually aims at minimizing a suitable empirical risk functional over the parameters of the network. Variants of (stochastic) gradient descent algorithms are then used for the practical optimization. Due to the non-convexity and high-dimensionality of this problem it is challenging to derive convergence guarantees for such algorithms.

In order to make theoretical progress, convergence properties for simplified models have been studied recently, most notably for \underline{linear} neural networks \cite{arora2018optimization,BRTW,Geo,Kohn2,ngraul24,saxe2013exact}. While linear networks are not powerful enough for most practical purposes due to limited expressivity (they can only represent linear functions), the convergence behaviour of gradient descent algorithms for learning them is still highly non-trivial because of the non-convexity of the corresponding empirical risk functional. Insights gained on linear models may inform theory for learning nonlinear neural networks.

In \cite{BRTW,ngraul24,arora2018optimization} gradient flow (GF) and gradient descent (GD) for learning linear \underline{fully connected} networks has been studied for the square loss. It is shown in \cite{BRTW,ngraul24} that GF and, under an upper bound on the step sizes, GD always converges to a critical point of the empirical risk. For almost all initialization, both (essentially) converge to a global minimizer, see \cite{BRTW,ngraul24} for the exact statements.
Additionally, \cite{BRTW} establishes that the corresponding flow of the network (the product of the individual weight matrices) can be interpreted as a Riemannian gradient flow.

The goal of this article is to extend convergence theory from linear fully connected neural networks to linear convolutional networks (LCN). 
In \cite{Geo} 
%by Kohn,  Merkh, Mont\'ufar and Trager, 
a mathematical formalism for LCNs was introduced. It was shown that functions represented by LCNs can be interpreted as certain polynomials and critical points of the loss landscape (using the square loss) were studied. In \cite{ackora25} the gradient flow for learning LCNs was studied and similarly to the fully connected case \cite{BRTW}, it was shown that the flow of function represented by the LCN can be interpreted as a Riemannian gradient flow. The main result of our article establishes that, under a mild non-degeneracy condition on the training data, this gradient flow always converges to a critical point of the empirical risk, defined via certain loss functions including the square loss.

%\subsection{Linear convolutional networks}

\subsection{Linear Convolutional Networks}
%\section{Linear Convolutional Networks}
\label{subsec:LCN}

%\subsection{Neural Networks}
A \textit{neural network} $f:\R^{d_0}\rightarrow\R^{d_N}$ of depth $N$ is  a composition of $N$ so-called \textit{layers}, each consisting of an affine function $\alpha_i: \R^{d_{i-1}}\rightarrow\R^{d_i}$, $\alpha_i(x) = \Wi x + b_i$, followed by a activation function $\sigma:\R \rightarrow\R$, which acts componentwise on a vector. 
Thus we can write 
$$
f(x)=\sigma\circ\alpha_N\circ\sigma\circ\alpha_{N-1}\circ\dots \circ \sigma\circ\alpha_{2}\circ\sigma\circ\alpha_1(x).
$$
Choosing the activation function $\sigma$ to be the identity and the bias vectors $b_i = 0$, $i\in \ul{N}:=\{1,\ldots,N\}$, results in a \textbf{linear} neural network 
of the form
%we mean a neural network where the activation function $\sigma$ is the identity and  where there are no bias vectors, i.e. $b_i = 0$ for all $i \in \ul{N}:=\{1,\ldots,N\}$. In this case, the neural network can be simplified as  
\begin{align*} f(x)&=\alpha_N\circ \circ\alpha_{N-1}\circ\dots \circ \circ\alpha_{2}\circ\alpha_1(x) = \WI{N}\cdot\cdots\cdot \WI{1}x = \W x 
\end{align*} 
where $\W = \WI{N}\cdot\ldots\cdot \WI{1} \in \R^{d_N \times d_0}$ represents the product weight matrix representing the neural network as a function.

%\subsection{Convolutions} 
%\label{sec_Convolutions} 

A \textit{linear convolutional network} (LCN) is a special case of a linear neural network, where each weight matrix
$\WI{i}$ represents a convolutional mapping. 
Such convolutions encode mappings that are equivariant with respect to translations and therefore allow effective processing of images. Since a convolution can be represented by a short vector, much fewer parameters have to be learned from training data, reducing the amount of required data.

We now introduce some further notions on LCNs closely following \cite{Geo}.
A convolutional layer $\alpha_i: \R^{d_{i-1}} \to \R^{d_i}$ is commonly defined in a machine learning context as
\begin{align*} &\qquad &(\alpha_i(x))_j=\sum_{n\in\ul{k_i}}\wi_n\cdot x_{(j-1)\cdot s_i+n} & &\text{for } j\in\ul{d_{i}}, \; x \in \R^{d_i}
\end{align*} 
with \textit{filter} $\wi=(\wi_{1},\wi_{2},\dots , \wi_{k_i})\in\R^{k_i}$, \textit{width} $k_i\in\N$ and \textit{stride} $s_i\in\N$. 

It is helpful to convert the affine mappings into matrix notation. Let us consider the example of a convolution $\alpha:\R^{8}\to\R^{3}$ with filter $w=(w_1,w_2,w_3,w_4) $ of width $k=4$ and stride $s=2$. Then $\alpha(x)=Wx$ with 
\begin{align} 
\label{eq_ExampleConvMa} 
W=\left(
\begin{array}{ccccccccc} w_1 & w_2 & w_3 & w_4 & & &  &   \\ & & w_1 & w_2 & w_3 & w_4 & &  \\  & & & & w_1 & w_2 & w_3 & w_4 \end{array}
\right)\in\R^{3\times 8}. 
\end{align} 
We call such a matrix \textit{convolutional matrix} or say that $W$ is convolutional. Let $\Pi_{(d_1,d_2),k,s}:\R^k\to\R^{d_2\times d_1}, w\mapsto W$ be the parameterization mapping that maps a filter $w$ of width $k$ and step size $s$ to the corresponding convolution matrix $W$, analogously to \eqref{eq_ExampleConvMa}. We restrict ourselves to convolutions without padding, i.e., we assume that the filter $w$ can be applied to $x$ without exceeding $x$. To ensure this, it is always assumed in the sequel that for $x\in\R^{d_0}$ and $\alpha_i :\R^{d_{i-1}}\to\R^{d_i}$ that 
\begin{align} 
\label{eq_dldlMinus1} 
d_i=\frac{d_{i-1}-k_i}{s_i}+1,\qquad i\in \ul{N} . 
\end{align} 
We call $(\textbf{d},\textbf{k},\textbf{s}): =\left((d_0,\ldots , d_N),(k_1,\ldots , k_N),(s_1,\ldots , s_N)\right)$  the architecture of the network (architecture for short) of depth $N$ if $(\textbf{d},\textbf{k},\textbf{s})$ satisfies  \eqref{eq_dldlMinus1}. Note that due to this equation, an architecture is uniquely defined by $(\textbf{d}_0,\textbf{k},\textbf{s})$.
%because of \eqref{eq_dldlMinus1}. 
The set of all possible LCNs with respect to a fixed architecture \archi{} is characterized by the set of matrices 
$$ \mathcal{M}_{\archi}=\left\{ W\in\R^{d_N\times d_0}\,\Big|\, W=\prod_{i=1}^N \Pi_{(d_{i-1},d_{i}),k_i,s_i}(\wi), \;  \wi\in \R^{k_i}, i \in \ul{N} \right\}. $$ 
We collect all filter vectors of an LCN in a vector $\wall:=\begin{pmatrix} w^{(1)} \\ \vdots \\ w^{(N)} \end{pmatrix}\in \R^{\sum_{i=1}^{N}k_i}$  and  define the mapping 
\begin{align*} \Pi_{\archi}:\R^{\sum_{i=1}^{N}k_i}\to\mathcal{M}_{\archi},\, (\wNs)\mapsto \prod_{i=1}^N \Pi_{(d_{i-1},d_{i}),k_i,s_i}(\wi).
\end{align*} 
The next proposition shows that a LCN is again a convolution and thus in particular all elements in $\mathcal{M}_{\archi}$ are convolutional. %The theorem including the proof comes from \cite[Prop. 2.2]{Geo}.\medskip

\begin{proposition}[{{\cite[Proposition 2.2]{Geo}}}]
\label{th_CompConv}
The composition of $N$ convolutions $\alpha_1,\ldots,\alpha_N$ with filter sizes $\textbf{k}=(k_1,\ldots , k_N)$ and strides $\textbf{s}=(s_1,\ldots , s_N)$ is a convolution $\alpha_v$ with stride $\s$ and filter width $\k$, and filter given by
\begin{align}
    \s & =\prod_{i=1}^N s_i \quad \text{and} \quad \k=k_1+\sum_{i=2}^N(k_i-1)\prod_{m=1}^{i-1}s_m. 
    \notag%\\
%\label{eq_endfilter}
%    \w_{m} & = \sum_{\substack{n\in \ul{k_N} ,\, l\in\ul{\tilde{k}}\\\ (n-1)\cdot \tilde{s}+l}=m} w^{(N)}_{n}\cdot \tilde{\w}_{l}.
\end{align} 
\end{proposition}

%\noindent 
For completeness and later reference, we include the proof from \cite{Geo} in the appendix. % cf. \cite{Geo}.

Using matrix notation for convolutions, we obtain
\begin{align*}
    \alpha_N \circ \ldots \circ \alpha_1 (x)= \prod_{i=1}^N \Pi_{(d_{i-1},d_{i}),k_i,s_i}(\wi)\cdot x =: \Pi_{\archi}(\wall) \cdot x.
\end{align*}
By \Cref{th_CompConv}, the matrix $\W:=\Pi_{\archi}(\wall)$ is (again) convolutional. The filter vector $\w$ belonging to $\V$ is called the final filter of the LCN. The mapping that maps $\wall$ directly to the final filter $v$ is called $\pi_{\archi}$, more precisely 
$${\displaystyle\pi_{\archi}:\R^{\sum_{i=1}^{N}k_i}\to\R^{\k}} \text{ with }v=\pi_{\archi}(\wall)$$ and $\k$ according to \Cref{th_CompConv}.\medskip

For an architecture \archi, the image of  $\pi_{\archi}$ is in general not  the whole space $\R^{\k}$, i.e., $\mathcal{M}_{\archi}\not=\mathcal{M}_{\l \l d_0,d_N\r,\k,\s\r }$. Architectures \archi{} with $\mathcal{M}_{\archi}=\mathcal{M}_{\l \l d_0,d_N\r,\k,\s\r }$ are referred to as \textit{filling architecture} in \cite{Geo}. According to \cite[Theorem 4.1]{Geo}, examples of \textit{filling architectures} are architectures \archi{} with $s=(1,\ldots,1)$ and $k=\l k_1\ldots,k_N\r$, where at most one $k_i$ is even. Consequently, 
%in filling architectures, if you are only interested in the final convolution, it is enough to consider 
for filling architectures it is often sufficient to consider the set $\mathcal{M}_{\l \l d_0,d_N\r,\k,\s\r }$. But we note that since, for $i,j\in\ul{N}$ with $ i\not =j$ and $\kappa\not=0$,
\begin{align*}
    \pi_{\archi}(\wall)=\pi_{\archi}\l\l \wI{1}, \ldots, \frac{1}{\kappa}\wi,\ldots,\kappa \wI{j},\ldots\r\r, %\text{ for } i,j\in\ul{N} \text{ with } i\not =j\text{ and } \kappa\not=0 ,
\end{align*} 
  the vector $\wall\in\R^{\sum_{i=1}^{N}k_i}$ is not uniquely determined by the vector $v=\pi_{\archi}(\wall)$.

%\bigskip

\subsection{Learning via gradient flow}
\label{sec_Learning-Task}
The task of learning consists in adapting the weights of a neural network to data. Given training labeled data $(x_{1},y_{1}),\ldots, (y_{m},x_{m})\in\R^{d_x} \times \R^{d_y}$ %and $y_{1},\ldots, y_{m}\in\R^{d_y}$ 
In our setup, we would like to find a LCN $f$ parameterized by filters $\wall$ such that $f(x_{j})\approx y_{j}$ for $j\in\ul{m}$. Given a (differentiable) loss function $\ell:\R^{d_y}\times \R^{d_y}\to\R$ and an architecture \archi{} of depth $N$ with $d_x=d_0$ and $d_y=d_N$, we introduce the empirical risk function
\begin{align*}
    \Loss: \R^{\sum_{i=1}^{N} k_i}\to \R,\wall\mapsto \sum_{i=1}^{m}\ell\left(\Pi_{\archi}(\wall)\cdot x_i,y_i\right),
\end{align*} 
where we use the shorthand notation $X=(x_1,\dots,x_m)\in\R^{d_0\times m}$ and $Y=(y_1,\dots,y_m)\in\R^{d_N\times m}$ for the data matrices.

The learning task can then be formulated as finding a (global) minimizer of the problem
\begin{align}
\label{eq_OptiProb2}
    \min_{ \wall\in\R^{\sum_{i=1}^{N}k_i}}\Loss(\wall).
\end{align}

%Classical methods for solving such a minimization problem, such as gradient descent or solving the gradient system, generally only converge to critical points of $\Loss{}$. 
In practice, variants of (stochastic) gradient descent are usually used for this minimization problem. For the purpose of a simplified mathematical analysis, however, it is common to use the abstraction of gradient flow. Then gradient descent can be viewed as an Euler discretization of the flow. In mathematical terms, we will focus on the gradient flow, i.e., the initial value problem 
\begin{align}
\label{eq_GradSys1}
    \frac{d\wall(t)}{dt}=-\nabla\Loss\l\wall(t)\r, \quad \wall(0) = \wall_0. 
\end{align}

\subsection{Main result}

Due to the non-convexity of the empirical risk $\Loss$ in the case of multiple layers, $N\geq 2$, convergence properties of the gradient flow are non-trivial. We establish that under certain assumptions on the loss function $\ell$, the gradient flow always converges to a critical point of $\Loss$. For simplicity of the exposition, we state our result here for the special case of the square loss
\begin{equation}\label{def:square-loss}
\ell : \R^{d_y} \times \R^{d_y} \to \R, \quad \ell(y,z) = \|y-z\|_2^2.
\end{equation}
%and refer to Theorem~\ref{} for the general case.

\begin{theorem} For the square loss \eqref{def:square-loss} and arbitrary initialization $\wall_0$, the gradient flow $\wall(t)$ defined in \eqref{eq_GradSys1} converges to a critical point of $\Loss$ as $t \to \infty $.
\end{theorem}
The general version of this result is stated in 
Theorem~\ref{th_Main} and applies also to other loss functions such as the $\ell_p$-loss for even $p \in 2\N$, the pseudo-Huber loss, the generalized Huber loss and the $\log$-$\cosh$-loss. 

The crucial ingredient in the proof of this theorem is to show boundedness of $\wall(t)$, which is non-trivial and will require the algebraic formalism introduced in \cite{Kohn2}. Once boundedness is established the result follows from the Lojasiewicz theorem.

\subsection{Structure of the paper.}

In Section~\ref{sec:conv-poly} we recall the relation of linear convolutional neural networks and polynomials from \cite{Kohn2} and prove some auxiliary statements.
Section~\ref{sec:grad-flow} presents our main results on the convergence of the gradient flow to critical points of the empirical risk.

\subsection{Acknowledgement}

H.R.\ and U.T.\ acknowledge funding by the Deutsche Forschungsgemeinschaft (DFG, German Research Foundation) -- project number 442047500 --
through the Collaborative Research Center "Sparsity and Singular Structures" (SFB 1481).

\section{Convolutions and Polynomials}
\label{sec:conv-poly}
  
For the proof of our main result we will consider the filters of an LCN as coefficient vectors of certain polynomials, as promoted in \cite{Kohn2}. These polynomials are chosen so that the composition of the convolutions is compatible with polynomial multiplication. This allows us to transfer statements about coefficients of polynomials to the filter vector of a convolution.

\subsection{Construction of polynomials using filter vectors}

Given a vector $w\in\R^{n+1}$, we can associate  a polynomial with $w$ via the mapping 
\begin{align*}
    p_{n}:\R^{n+1}\to\R[x]_{n},\; w\mapsto \sum_{i=0}^{n}w_{i+1}x^{i}.
\end{align*}

In the next theorem, we show that we can factorize the polynomial $p_{\k-1}(v)$ of the final filter $v$ into $N$ polynomials, where the coefficients of these polynomials either correspond to the entries of \wall or are zero. The following proposition is based on \cite[Proposition 2.6 and Remark 2.8]{Geo}. For convenience we give a detailed proof.
\medskip

\begin{proposition}[{{\cite[Proposition 2.6 and Remark 2.8]{Geo}}}]
\label{th_Productps}
Let $(\textbf{d},\textbf{k},\textbf{s})$ be an architecture of depth $N$. Let \wall be the vector of all filter vectors of an LCN and define $v:=\pi_{\archi}(\wall)$ with associated filter width $\k$. Then the polynomial
$p_{\k-1}(v)$ can be decomposed into $N$ polynomials $\tilde{p}_{k_i,\textbf{s}}(\wi)$ with respective degree $t_i=(k_i-1)\prod_{n=1}^{i-1}s_n$. These polynomials are given by 
\begin{align}
\label{eq_SpecialPolynomConstruction}
    \Tilde{p}_{k_i,\textbf{s}}(\wi)(x)=\sum_{j=0}^{k_i-1}\wi_{j+1}\cdot x^{j\prod_{n=1}^{i-1}s_n} \qquad\qquad\text{for }i\in\ul{N}.
\end{align}
\end{proposition}
\begin{proof}
We prove the statement by induction on $N$. For $N=1$, the first filter \wI{1} corresponds to the final filter, thus
\begin{align*}
\Tilde{p}_{k_1,\textbf{s}}(\wI{1})(x)=\sum_{j=0}^{k_1-1}\wI{1}_{j+1}\cdot x^{j}=p_{k_1-1}(\wI{1})(x).
\end{align*}

We assume now the claim holds for $N-1$. In the setup of the proposition for $N-1$, according to  \Cref{th_CompConv}, $\widetilde{\w}=\pi_{\archi}\l(\wI{1},\dots, \wI{N-1})\r$ is a convolution with filter width $k_{ \widetilde{v}}=k_1+\sum_{i=2}^{N-1}(k_i-1)\prod_{m=1}^{i-1}s_m$ and stride $s_{ \widetilde{v}}=\prod_{i=1}^{N-1} s_i$. Let again $\w=\pi_{\archi}(\wI{1},\dots,\wI{N})$ with filter width $\k$ and stride $\s$. Then, according to the induction hypothesis
\begin{align*}
    \prod_{j=1}^{N} \Tilde{p}_{k_j,\textbf{s}}(w^{(j)})(x)& =  \Tilde{p}_{k_N,\textbf{s}}(w^{(N)})(x) \cdot p_{k_{ \widetilde{v}}-1}(\widetilde{v})(x)\\ 
    & = \left(\sum_{j=0}^{k_{N}-1}\wI{N}_{j+1}\cdot x^{j\prod_{n=1}^{N-1}s_n}\right)\cdot \left(\sum_{i=0}^{k_{ \widetilde{v}}-1}\widetilde{v}_{i+1}\cdot x^{i}\right) \\
    & = \sum_{j=0}^{k_{N}-1}\sum_{i=0}^{k_{ \widetilde{v}}-1}\wI{N}_{j+1}\cdot\widetilde{v}_{i+1}\cdot x^{j\cdot s_{ \widetilde{v}}+i} 
\end{align*}
is a polynomial of degree $(k_{N}-1)s_{ \widetilde{v}}+k_{ \widetilde{v}}-1=(k_{N}- 1)s_{ \widetilde{v}}+k_1+\sum_{i=2}^{N-1}(k_i-1)\prod_{m=1}^{i-1}s_m-1
%\stackrel{ \Cref{th_CompConv}}{=}
=\k-1$ by \Cref{th_CompConv}. This allows us to further rewrite the product as %:}%%%%intertext end
\begin{align*}
    \prod_{j=1}^{N} \Tilde{p}_{k_j,\textbf{s}}(w^{(j)})(x)=& \sum_{m=0}^{\k-1}x^m\sum_{\substack{j\in \ul{k_N} ,\, i\in\ul{k_{ \widetilde{v}}}\\ m=(j-1)\cdot s_{ \widetilde{v}}+i-1}}\wI{N}_{j}\cdot\widetilde{v}_{i}\\
    & = %\stackrel{\eqref{eq_endfilter}}{=} 
    \sum_{m=0}^{\k-1}x^m\cdot \v_{m+1}
    = p_{\k-1}(\w)(x),
\end{align*}
where we used \eqref{eq_endfilter} from the proof of Proposition~\ref{th_CompConv} in the last line.
\end{proof}
\medskip

Let $\wI{1}\in \R^{k_1}$ and $\wI{2}\in \R^{k_2}$ be two filter vectors for the convolutions $\alpha_1$ and $\alpha_2$ with strides $s_1$ and $s_2$. Then, according to Proposition~\ref{th_Productps}, the filter vector of $\alpha_2\circ\alpha_1$ corresponds to the coefficients of the polynomial $\Tilde{p}_{k_1,(s_1,s_2)}(\wI{1})(x)\cdot\Tilde{p}_{k_2,(s_1,s_2)}(\wI{2})(x)$. In this way, the convolution is compatible with the multiplication of polynomials constructed according to \eqref{eq_SpecialPolynomConstruction}. 
\medskip

In particular, for an architecture with $s_i=1$ for all $i\in\ul{N}$, according to \Cref{th_Productps}, we have
\begin{align*}
    p_{\k-1}(\w)(x)= p_{k_{1}-1}(\wI{1})(x)\cdot\ldots\cdot p_{k_{N}-1}(\wI{N})(x).
\end{align*}
Let $\C[x]_n$ be the set of polynomials in $x$ with complex coefficients whose degree is at most $n$. For a polynomial $q(x)=a_nx^n+a_{n-1}x^{n-1}+\cdots + a_0\in \C[x]_n$, let us define the $p$-norm for $p\geq 1$ by 
\[
%\Vert q(x)\Vert_p
\triple{q}_p
:= \|a\|_p =
\left(\sum_{i=0}^n\vert a_i\vert^p\right)^{\frac{1}{p}}.
\]
Then the $p$-norm of $p_{\k-1}(\w)$ equals the $p$-norm of $\w$, i.e., 
$\triple{p_{\k-1}(\w)}_p = \|\w\|_p$
%more precisely $\Vert p_{\k-1}(\w)(x) \Vert_p =\Vert \w \Vert_p$ 
and $
\triple{\Tilde{p}_{k_i,\textbf{s}}(\wi)}_p =\Vert \wi \Vert_p$ for $i\in\ul{N}$.

\subsection{Auxiliary statements about polynomials}

This subsection contains general statements about polynomials that are needed for the proofs in the next chapter. The focus is on the relationship between the boundedness of the roots of a polynomial and the boundedness of the coefficients of the polynomial. This is helpful because, according to the polynomial construction in the previous chapter, the polynomial coefficients correspond to the components of a filter vector.

We introduce $\ul{n}_0 := \{0\} \cup \ul{n} 
= \{0,1,\hdots,n\}$.

\begin{lemma}
\label{lem_NormSubmulti}
    The 1-norm on the set of complex polynomials %of degree $n\in \N_0$ 
    is submultiplicative, i.e., for two complex polynomials $p,q$ %\in\C[x]_n$ with $\deg(p)+\deg(q)\leq n$ 
    we have  
    \begin{align*}
        \triple{pq}_1 \leq
        \triple{p}_1  \triple{q}_1.
        %\left\Vert p(x)q(x)\right\Vert_1\leq   
        %\left\Vert p(x)\right\Vert_1 \left\Vert q(x)\right\Vert_1.
    \end{align*}
\end{lemma}
\begin{proof} Let $n\in \N_0$ be the degree of $pq$. Then $p$ and $q$ have degree at most $n$ and we can write 
     $p(x)=a_nx^n+a_{n-1}x^{n-1}+\cdots + a_0$ and $q(x)=b_nx^n+b_{n-1}x^{n-1}+\cdots + b_0$. %be polynomials in $\C[x]_n$ with $\deg(p)+\deg(q)\leq n. $ %We define $\ul{n}^*:=\ul{n}\cup \{0\}=\{0,1,\dots,n\}$. 
    Then 
    \begin{align*}
        & p(x)\cdot q(x)=\sum_{k=0}^nx^k\sum_{\substack{i,j\in\ul{n}_0 \\ i+j=k}}a_ib_j=\sum_{k=0}^nx^k\sum_{i=0}^ka_ib_{k-i}
    \end{align*}
        %\intertext{and we get}
    and we obtain
    \begin{align*}
    \triple{pq}_1  
        %\left\Vert p(x)\cdot q(x)\right\Vert_1
        =\sum_{k=0}^n\left|\sum_{i=0}^ka_ib_{k-i}\right|\leq 
\sum_{k=0}^n\sum_{i=0}^k\left|a_i\right|\left|b_{k-i}\right| \leq \sum_{i=0}^n |a_i| \sum_{j=0}^n |b_{j}| =
\triple{p}_1   \triple{q}_1. %\left\Vert q(x)\right\Vert_1.
\end{align*}
This completes the proof.
\end{proof}
\medskip

For $p>1$ the $p$-norm is not submultiplicative. However, due to Lemma \Cref{lem_NormSubmulti} and since any two norms on the  finite dimensional vector space $\C[x]_n$ are equivalent, there is a constant $c(p,n)>0$  such that for $q_1 q_2 \in \C[x]_n$ we have
$$\triple{q_1 q_2}_p \leq c(p,n) \triple{q_1}_p \triple{q_2}_p.
%\left\Vert q_1(x)\right\Vert_p\left\Vert q_2(x)\right\Vert_p.
$$
%\medskip

The next lemma shows that the roots of a polynomial $p(x)$ can be bounded in terms of the coefficients of the polynomial.\medskip

\begin{lemma}
\label{lem_rootsBound}
Let $p(x)=a_nx^n+a_{n-1}x^{n-1}+\cdots + a_0 \in \C[x]_n$ be a polynomial of degree $n\in\N$. %with $a_i\in\C$.  
Set $A:=\frac{n}{\vert a_n\vert}\max\limits_{i\in\ul{n}_0}\vert a_i\vert$. If $x \in \C$ is an arbitrary zero of $p$, i.e., $p(x) = 0$, then $|x| \leq A$.
%then $x\in\C$ holds:
%$$
%p(x)=0 \Longrightarrow \vert x\vert \leq A.
%$$
\end{lemma}

\begin{proof}
We proceed by contradiction and assume that $x \in \C$ is a zero of $p$ with $\vert x\vert > A\geq n\geq 1$. For $k\in\ul{n}$ and $a_{n-k}\not=0 $ we have  
\begin{align}
\label{Ab1}
\frac{\vert a_{n-k}\vert}{\vert a_n\vert \vert x^k\vert}\stackrel{\vert x \vert\geq 1}{\leq} \frac{\vert a_{n-k}\vert}{\vert a_n\vert \vert x\vert}\stackrel{\vert x\vert> \frac{n\vert a_{n-k}\vert }{\vert a_n \vert} }{<} \frac{\vert a_{n-k}\vert}{\vert a_n\vert }\frac{\vert a_n\vert}{n\vert a_{n-k}\vert} = \frac{1}{n}.
\end{align}
For $a_{n-k}=0$, we also have $\frac{\vert a_{n-k}\vert}{\vert a_n\vert \vert x^k\vert}<\frac{1}{n}$.

We obtain
\begin{alignat*}{2}
\vert p(x) \vert &= \vert a_nx^n \vert\cdot \vert 1+\frac{a_{n-1}}{a_n x}+\dots +\frac{a_{0}}{a_n x^n}\vert\\ &\geq \vert a_nx^n \vert\cdot \l 1- \vert\frac{a_{n-1}}{a_n x}+\dots +\frac{a_{0}}{a_n x^n}\vert \r\\ &\stackrel{\eqref{Ab1}}{>}\vert a_nx^n \vert\cdot \l 1- n\cdot \frac{1}{n} \r=0, %\tag*{\qedhere}
\end{alignat*}
where the triangle inequality was used in the penultimate inequality.
\end{proof}

\medskip

According to this lemma, the roots of a polynomial are bounded in terms of the coefficients of the polynomial. However, if the leading coefficient is arbitrarily small, the bound of the roots is arbitrarily large. However, the next lemma states that we can  write any polynomial with complex coefficients as a product of linear factors such that the $1$-norms of any of these linear factors is controlled in terms of the $1$-norm of the polynomial. 
\medskip

\begin{lemma}
\label{lem_PolyZerLinFak}
Let $p(x)=a_nx^n+a_{n-1}x^{n-1}+\cdots + a_0$ be a polynomial of degree $n\in\N$ with $a_i\in\C$ and  let $T:=\max\{\triple{p}_1,1\}$. Then there exists a factorization of $p$ into linear factors $u_1,\dots,u_n$ in $\C[x]_1$ such that 
\begin{align*}
p(x)=\sgn(a_n)\cdot\displaystyle\prod_{i=1}^{n}u_i(x) \quad\text{and}\quad \triple{u_i}_1\leq 6nT\quad \text{for } i\in\ul{n}. 
\end{align*}
%where $c$ is a constant that depends only on $n$. 
\end{lemma}
\begin{proof} The statement is clear for $n=1$, so let us assume that $n\geq 2$.
Since $a_n\neq 0$ the fundamental theorem of algebra implies that there exist $n$ complex roots $z_1,\dots,z_n$ with $z_i\in\C$ and
\begin{align*}
    p(x)=a_n\prod_{i=1}^n (x-z_i)=a_n\l x^n+\frac{a_{n-1}}{a_n}\cdot x^{n-1}+\dots +\frac{a_1}{a_n}\cdot x+\frac{a_0}{a_n}\r.
\end{align*}

\textbf{Case 1:} If $\vert z_i \vert\leq 1$ for all $i\in\ul{n}$ then we have $\vert a_n\cdot z_i \vert\leq\vert a_n \vert\leq T $ for $i\in\ul{n}$, hence $\triple{u_i}_1 \leq 2 T$ and the assertion follows. 

\textbf{Case 2:} Suppose that there exists an $i\in\ul{n}$ with $\vert z_i \vert> 1$ and without loss of generality
assume that $\vert z_1 \vert\geq\vert z_2 \vert\geq \vert z_3 \vert\geq \dots \geq \vert z_n \vert$. Then there exists an $r\in\N$ such that $\vert z_i\vert\geq 1$ for $i\leq r$ and $\vert z_i\vert< 1$ for $i> r$. Let $B:=\frac{1}{\vert a_n\vert}\displaystyle\max_{i\in\ul{n}_{0}}\vert a_i\vert \geq 1$. We will show that the product of the $r$ largest roots $\vert z_1\cdot\dots\cdot z_r\vert$ is upper bounded by $3Bn^r$. The $(n-r)$th coefficient of the polynomial $x^n+\frac{a_{n-1}}{a_n}\cdot x^{n-1}+\dots +\frac{a_1}{a_n}\cdot x+\frac{a_0}{a_n}$ satisfies
\begin{align}
\label{eq_rootsInequalityProof}
    \left|\frac{a_{n-r}}{a_{n}}\right|&= \left|\l-1\r^{r}\cdot\sum_{j_1<\ldots <j_r} z_{j_1}\cdot\ldots\cdot z_{j_r}\right|\notag\\
    &=\left|z_{1}\cdot\ldots\cdot z_{r}+\sum_{i=1}^{n-r}z_{r+i}\sum_{\substack{j_1<\ldots <j_{r-1}\\j_k\not\in\{r+1,\ldots,r+i\} \forall k}} z_{j_1}\cdot\ldots\cdot z_{j_{r-1}}\right|.
\end{align}
We decompose the right-hand summand further by applying the same argument repeatly,
{\small
\begingroup
\allowdisplaybreaks
\begin{align*}
  & \sum_{i=1}^{n-r}z_{r+i}\sum_{\substack{j_1<\ldots <j_{r-1}\\ j_k\not\in\{r+1,\ldots,r+i\}}}z_{j_1}\cdot\ldots\cdot z_{j_{r-1}} \\
    = & \sum_{i=1}^{n-r}z_{r+i}\left(\sum_{j_1<\ldots <j_{r-1} }z_{j_1}\cdot\ldots\cdot z_{j_{r-1}}-\sum_{\substack{j_1<\ldots <j_{r-1}\\\ \exists j_k : j_k\in\{r+1 ,\ldots,r+i\}}}z_{j_1}\cdot\ldots\cdot z_{j_{r-1}}\right) \\
    =\, & \sum_{i=1}^{n-r}z_{r+i}\left(\l -1\r^{r-1}\cdot\frac{a_{n-(r-1)}}{a_n}-\sum_{i_2=1}^{i}z_{r+i_2}\sum_{\substack{j_1<\ldots <j_{r-2}\\ j_k\not\in\{r+1, \ldots,r+i_2\}\forall k}}z_{j_1}\cdot\ldots\cdot z_{j_{r-2}}\right)      \\
  & \, \vdots \\
% \begin{split}
  =  & \sum_{i=1}^{n-r}z_{r+i}\Bigg( (-1)^{r-1}\cdot\frac{a_{n-(r-1)}}{a_n}-\sum_{i_2=1}^{i}z_{r+i_2} \Bigg( ( - 1)^{r-2}\cdot \frac{a_{n-r+2}}{a_n}\\
    & \quad -\sum_{i_3=1}^{i_2}z_{r+i_3}\Bigg( ( -1)^{r-3}\cdot\frac{a_{n-r+3}}{a_n} \ldots\Bigg(\ldots-\sum_{i_{r-1}=1}^{i_{r-2}}z_{r+i_{r-1}}\Bigg(-\frac{a_{n-1}}{a_n} \\
    & \quad - \sum_{i_r=r+1}^{r+i_{r-1}}z_{i_r} \Bigg)\Bigg)\ldots\Bigg).
%\end{split}
\end{align*}
\endgroup
}Using that $\sum_{i_{r-(k-1)}=1}^{i_{r-k}}1\leq n$ for $1\leq k\leq r-1$ this implies that
{\footnotesize
\begin{align*}
      & \Biggl|\sum_{i_{1}=1}^{n-r}z_{r+i_{1}}\sum_{\substack{j_1<\ldots <j_{r-1}\\ j_k\not\in\{r+1,\ldots,r+i_{1}\}}}z_{j_1}\cdot\ldots\cdot z_{j_{r-1}}\Biggr| \\
\begin{split}
  \leq \, & \sum_{i_{1}=1}^{n-r}\underbrace{\vert z_{r+i_{1}}\vert}_{\leq 1}\Bigg(\underbrace{\Big\vert\frac{a_{n-(r- 1)}}{a_n}\Big\vert}_{\leq B}+\sum_{i_2=1}^{i_{1}}\underbrace{\vert z_{r+i_2}\vert}_{\leq 1}\Bigg( \underbrace{\Big\vert\frac{a_{n- r+2}}{a_n}\Big\vert}_{\leq B}+\sum_{i_3=1}^{i_2}\underbrace{\vert z_{r+i_3}\vert}_{\leq 1}\Bigg(\underbrace{\Big\vert\frac{a_{n-r+3}}{a_n}\Big\vert}_{\leq B}\\\
    & \quad +\ldots\Bigg(\ldots+\sum_{i_{r-1}=1}^{i_{r-2}}\underbrace{\vert z_{r+i_{r-1}}\vert}_{\leq 1}\Bigg(\underbrace{\Big\vert\frac{a_{n- 1}}{a_n}\Big\vert}_{\leq B}+\underbrace{\sum_{i_{r}=r+1}^{r+i_{r-1}}\underbrace{\vert z_{i_r}\vert}_{\leq 1}}_{\leq n}\Bigg)\Bigg)\Bigg)\ldots\Bigg)
\end{split}\\\
\leq \, & \sum_{i_{1}=1}^{n-r} \Bigg(B+\sum_{i_2=1}^{i_{1}} \Bigg( B+\sum_{i_3=1}^{i_2} \Bigg(B
     +\ldots\Bigg(\ldots+\sum_{i_{r-2}=1}^{i_{r-3}} \Bigg(B+\underbrace{\sum_{i_{r-1}=1}^{i_{r-2}}\Bigg(B+n\Bigg)}_{\leq Bn + n^2}\Bigg)\Bigg)\ldots\Bigg)\\
\leq \, & \sum_{i_{1}=1}^{n-r}\Bigg(B+\sum_{i_2=1}^{i_{1}}\Bigg( B+\sum_{i_3=1}^{i_2}\Bigg(B +\ldots\Bigg(\ldots +\sum_{i_{r- 3}=1}^{i_{r-4}}\Bigg(B+
%\underbrace{
\sum_{i_{r-2}=1}^{i_{r-3}}\Bigg(B(1+n)+n^2)\Bigg)%}
%_{\leq An(n+1)}
\Bigg)\ldots\Bigg)\\  
\leq &\sum_{i_1=1}^{n-r}\left( B(1+n+\cdots + n^{r-2}) + n^{r-1}\right)\\
\leq &B(n+n^2+\cdots + n^{r-1}) + n^{r} \leq 
(B+1)n^r.
%\leq \, & \sum_{i_{1}=1}^{n-r}\Bigg(\frac{A}{n}+\sum_{i_2=1}^{i_{1}}\Bigg( \frac{A}{n}+\sum_{i_3=1}^{i_2}\Bigg(\frac{A}{n} +\ldots\Bigg(\ldots+\underbrace{\sum_{i_{r- 4}=1}^{i_{r-5}}\Bigg( \frac{A}{n}+\sum_{i_{r-3}=1}^{i_{r-4}}An(n+1)}_{\leq A...}\Bigg)\ldots\Bigg).
\end{align*}}

This inequality implies an upper bound of the product of the $r$ largest roots in the following way,
\begin{align}
\label{eq_ProdziKleinerOA}
 & \Bigg| z_{1}\cdot\ldots\cdot z_{r}\Bigg|- \Bigg|\sum_{i=1}^{n-r}z_{r+i}\sum_{\substack{j_1<\ldots <j_r\\j_k\not\in\{r+1,\ldots,r+i\}}} z_{j_1}\cdot\ldots\cdot z_{j_{r-1}}\Bigg|\notag\\
 \leq \;& \Bigg| z_{1}\cdot\ldots\cdot z_{r}+\sum_{i=1}^{n-r}z_{r+i}\sum_{\substack{j_1<\ldots <j_r\\j_k\not\in\{r+1,\ldots,r+i\}}} z_{j_1}\cdot\ldots\cdot z_{j_{r-1}}\Bigg| \; \stackrel{\eqref{eq_rootsInequalityProof}}{=} \; \left| \frac{a_{n-r}}{a_{n}}\right| \leq B.
 \end{align}
Therefore
\begin{align}
 \left| z_{1}\cdot\ldots\cdot z_{r}\right| \leq & B + \Bigg|\sum_{i=1}^{n-r}z_{r+i}\sum_{\substack{j_1<\ldots <j_r\\j_k\not\in\{r+1, \ldots,r+i\}}} z_{j_1}\cdot\ldots\cdot z_{j_{r-1}}\Bigg|
 \leq B+(B+1)n^{r}\\
 \leq & 3Bn^r \notag,
%|\stackrel{\eqref{eq_rootsInequalityProof2}}{\leq} 
 %\leq 2An^r .
\end{align}
where we used that $B \geq 1$.
We construct the linear factors $u_i$ as
\begin{align*}
   & u_i(x)=\frac{\vert a_n z_{1}\cdot\ldots\cdot z_{r}\vert^{1/r}}{\vert z_i\vert}(x-z_i) &&\text{for } 1\leq i\leq r \\
    \intertext{and} 
    & u_i(x)= (x-z_i) && \text{for } r< i\leq n.
\end{align*}
For $i>r$ we have $|z_i| \leq 1$ so that $\triple{u_i}_1 = 1$. 
%less than or equal to $1$ and are therefore bounded. 
For $1\leq i\leq r$, using that $|a_n| B \leq T$, $|z_i| \geq 1$ and $T\geq 1$, the coefficient of the linear term of $u_i$ satisfies 
\begin{align*}    &\big\vert\vert a_n z_{1}\cdot\ldots\cdot z_{r}\vert^{1/r} \frac{1}{\vert z_i\vert}
%_{\leq 1}\big\vert\stackrel{\eqref{eq_ProdziKleinerOA}}{
\leq (\vert a_n \vert 3Bn^r)^{1/r}
\leq 3 n T
%\stackrel{\vert a_n\vert A\leq nT}{\leq}
%\stackrel{T\geq 1}{\leq} 2Tn^2.
\end{align*}
The constant coefficient of $u_i$ can be estimated analogously 
\begin{align*}
   &\big\vert\vert a_n z_{1}\cdot\ldots\cdot z_{r}\vert^{1/r} \frac{z_i}{\vert z_i\vert}\big\vert=\vert a_n z_{1}\cdot\ldots\cdot z_{r}\vert^{1/r} \leq 3nT.
%\stackrel{\eqref{eq_ProdziKleinerOA}}{\leq} 
\end{align*}
This means that $\triple{u_i}_1 \leq 6nT$ for $i\in\ul{n}$. 
It is easy to check that indeed $p(x)=\sgn(a_n)\cdot\displaystyle\prod_{i=1}^{n}u_i(x)$. %This means that $\triple{u_i}_1 \leq 6nT$ for $i\in\ul{n}$ 
and thus the assertion follows.
\end{proof}
\medskip

\section{Gradient flows for training linear convolutional networks}
\label{sec:grad-flow}
In this section, we consider gradient flows of linear convolutional networks and study their convergence to a critical point of \Loss.
Our main technical tool will be
the Łojasiewicz gradient inequality, which reduces the problem to showing boundedness of the flow.

\subsection{ Łojasiewicz'  gradient inequality}

Łojasiewicz' gradient inequality \cite{eins} is a regularity condition on functions.

\begin{definition}[Łojasiewicz gradient inequality]
A differentiable function $\phi:\R^{n}\to \R$ satisfies the Łojasiewicz gradient inequality (LGI) in $x\in\R^n$ if there exists a neighborhood $U$ of $x$, and constants $c>0$ and $\mu\in[0,1)$ such that
\begin{align*}
    \vert \phi(x)-\phi(y) \vert^\mu \leq c \Vert \nabla \phi(y) \Vert \qquad \mbox{ for all }  y\in U,
\end{align*}
where $\Vert\cdot\Vert$ is some norm on $\R^n$.
\end{definition}
%\medskip
%Note that due to the equivalence of all norms on $\R^n$, if the LGI holds for some norm, it will hold for any norm on $\R^n$ for a suitable  constant $c$.  
 Łojasiewicz' theorem \cite{Loja1} states that analytic functions satisfy the Łojasiewicz gradient inequality at any point.
 %for arbitrary norms on $\R^n$. 
%An important conclusion for us from the fulfillment of the \LGU{} is provided by the following theorem.\medskip
An important consequence of the \LGU{} is stated in the next theorem, see e.g.\ \cite[Theorem 3.1]{BRTW} resp.\  \cite[Theorem 2.2]{eins}.

\begin{theorem}
\label{LGIbeschrKonvKritPunkt}
Let $f:\R^{n}\to\R$ be a differentiable function which satisfies \LGU{} at every point. If $t\mapsto x(t)\in\R^n, t\in[0,\infty)$,
solves the ODE $\dot{x}(t)=-\nabla f(x(t))$ and is bounded,
%is a bounded curve and a solution of the gradient system , 
then $x(t)$ converges to a critical point of $f$ as 
$t\to\infty$. 
\end{theorem}
%\begin{proof}
%For a proof, see e.g. \cite[Theorem 3.1]{BRTW} resp.  \cite[Theorem 2.2]{eins} which is the most important ingredient.
In \cite[Theorem 3.1]{BRTW} it is required that $f$ is analytic. Inspecting the proof, however, shows that thanks to  Łojasiewicz's theorem \cite{Loja1}  it is enough that $f$  is a differentiable function which satisfies the \LGU{} at every point.
%\end{proof}\medskip

%With this theorem we have  the most important tool to show the convergence of the gradient flow to a critical point of \Loss. 

%\bigskip

\subsection{Convergence of the gradient flow}

We consider linear convolutional networks of depth $N$ with a fixed architecture \archi{} and fixed training data $X$ and $Y$. We use the abbreviations $\Losskurz:=\Loss$ and $\pi:=\pi_{\archi}$ as well as $\ksum:=\sum_{i=1}^{N}k_i$.

We consider learning a convolutional network with parameters $\wall = (w^{(1)},\hdots,w^{(N)})$ via the gradient flow $\wallt = (w^{(1)}(t),\hdots,w^{(N)}(t))$ defined for some initialization $\wall_0$ via the ODE
\begin{align}
\label{eq_GradientSystemwi}
    \frac{d\wi(t)}{dt}=-\nabla_{\wi}\Losskurz\l\wNsv{t}\r \quad\text{with }\quad \wi(0)=\wi_0, \quad i\in\ul{N}.
\end{align}
Motivated by \Cref{LGIbeschrKonvKritPunkt}, in order to prove the convergence of \wallt{} to a critical point of \Losskurz, we aim at showing boundedness of \wallt. 
%We achieve this by showing the boundedness of \wit for all $i\in\ul{N}$, because $\Vert\wall(t)\Vert_1=\sum_{i=1}^{N}\Vert \wi(t)\Vert_1$. To show the boundedness of \wit 's, the following lemma and its proof from \cite{Geo} will be helpful. The statement in the original refers to architectures where all step sizes are 1, but since the choice of $\textbf{s}$ is irrelevant in the proof, arbitrary step sizes can be chosen.
The so-called balancedness as stated in the next result from \cite{Geo} is a crucial ingredient for achieving this.
The original statement in \cite{Geo} considers the special case that all strides $s_i$ are $1$, but the proof, which we include for convenience, works also for general strides.
%Again, for convenience, we include the proof.
%\medskip

\begin{lemma}[{{\cite[Proposition 5.13]{Geo}}}]
\label{lem:Normdiffconst}
Let $\Losskurz$ be continuously differentiable and let $(\wNsv{t})$ be the solution vector of the gradient system \eqref{eq_GradientSystemwi} on an interval $I$. Then, for each pair $i,j\in\ul{N}$ and $t\in I$, the difference $\delta_{ij}(t):=\Vert \wi(t)\Vert^{2}_{2}-\Vert \wI{j}(t)\Vert^{2}_{2}$ is constant in $t$.
\end{lemma}
\begin{proof}
Let for $i,j\in\ul{N}$ and $\kappa>0$
\begin{align*}
\tilde{\theta}_{ij}(\kappa,t):= & \left( \wI{1}(t)^\top,\ldots,\kappa \wi(t)^\top,\wI{i+1}(t)^\top,\ldots,\frac{1}{\kappa}\wI{j}(t)^\top,\right. \\
& \left. \;\;
\wI{j+1}(t)^\top,\ldots, \wI{N}(t)^\top\right)^\top. 
\end{align*}
Since the scaling parameter $\kappa$ has no influence on the final filter, we have 
$$
\dif{\kappa}\l\Losskurz\l\tilde{\theta}_{ij}(\kappa,t)\r\r\equiv 0 \qquad \mbox{ for all } i,j\in\ul{N}.
$$
In particular, we obtain
\begin{align}
\label{iszero}
\begin{aligned}
    0= &\l\dif{\kappa}\l\Losskurz\l\tilde{\theta}_{ij}(\kappa,t)\r\r\r\bigg|_{\kappa=1}\\
    = & \l\nabla\Losskurz\l\tilde{\theta}_{ij}(\kappa,t)\r\r^\top\Bigg|_{\kappa=1} \cdot \l\dif{\kappa}\tilde{\theta}_{ij}(\kappa,t)\r\Bigg|_{\kappa=1}\\
     = & \l\nabla\Losskurz\l\wall(t)\r\r^\top \cdot \l0,\ldots,0,\wi(t)^\top,0,\ldots,0,-\wI{j}(t)^\top,0\ldots\r^\top.
     \end{aligned}
\end{align}
Therefore, using also \eqref{eq_GradientSystemwi}, $\delta_{i,j}(t)=\Vert \wi(t)\Vert^{2}_{2}-\Vert \wI{j}(t)\Vert^{2}_{2}$ satisfies
\begin{align*}
   \dif{t} \delta_{ij}(t)&= \; 2\wi(t)^{\top}\cdot\dif{t}\wi(t)-2\wI{j}(t)^\top\cdot\dif{t}\wI{j}(t)\\
    &= \; 2\l0,\ldots,\wi(t)^{\top},0,\ldots,-\wI{j}(t)^{\top},0,\ldots\r\cdot\dif{t}\wall(t)\\
    & = %\stackrel{\eqref{eq_GradientSystem}}{=} 
    -2\l0,\ldots,\wi(t)^{\top},0,\ldots,-\wI{j}(t)^{\top},0,\ldots\r\cdot\nabla\Losskurz\l\wall\r = 0   % &\stackrel{\eqref{iszero}}{=} \; 0.
\end{align*}
This means that $\delta_{i,j}(t)$ is constant in $t$.
\end{proof}
%\medskip

The previous lemma implies that it is enough to show boundedness of $\wI{j}(t)$ for one value of $j$ in order to establish boundedness of all $w^{(1)}(t), \hdots, w^{(N)}(t)$.
%If we show that for one value of $j\in\ul{N}$ the expression $\Vert \wI{j}(t)\Vert_2^2$ is bounded, the boundedness of all $\wi(t)$ follows from the previous lemma. It is therefore sufficient to show that one of the $\wit$ is bounded to prove the boundedness of \wallt.
%\medskip

The following theorem provides sufficient conditions on $\Losskurz$ ensuring  that the gradient flow \wallt{} defined via \eqref{eq_GradientSystemwi} converges to a critical point of \Losskurz. 
%\medskip

\begin{theorem}
\label{th_Main}
Suppose that  $\Losskurz$ is twice continuously differentiable and satisfies the \LGU{} on $\R^{\ksum}$. For a given initialization $\wall_0\in\R^{\ksum}$, consider % \wallt solves 
the gradient system 
\begin{align}
\label{gradient flow}
\frac{d\wall(t)}{dt}=-\nabla\Losskurz\l\wall(t)\r \quad \text{with}\quad \wall(0)=\wall_0.
\end{align}
Then the following statements hold:
\begin{enumerate}
 \item There exists a maximal open interval $I$ with $0\in I$ and a (unique) %maximal 
 integral curve $\wallt: I\to \R^{\ksum}$ that solves  \eqref{gradient flow}. 
 \item For $I$ as above, suppose there exists a non-decreasing function $g:\R\to \R_{\geq 0}$ 
 %which is increasing on $\Losskurz(\wall(I))$ and 
 that satisfies  
 \begin{align}
\label{eq_EndfilterBoundedbyL}
\Vert\pi(\wall)\Vert_1\leq g\l\Losskurz\l\wall\r\r \quad\mbox{ for all } \wall.
%t\in I \text{ with } t\geq 0.
\end{align}
Then $\wall(t)$ is defined and bounded for all $t\geq 0$ and converges to a critical point of $\Losskurz$ as $t\to\infty$.
\end{enumerate}
\end{theorem}

The proof requires the following technical lemma.

\begin{lemma}\label{lem:bound-individual} Assume that $\wall = (w^{(1)},\hdots,w^{(n)})$ satisfies $\|\pi(\wall)\|_1 \leq T$ and set $\delta_i = \|w^{(i+1)}\|_2^2 - \|w^{(i)}\|_2^2$. Then there exists a constant $\tau = \tau(T, k_v, \delta_1,\hdots,\delta_{N-1})$ such that
\[
\|w^{(i)}\|_2^2 \leq \tau \quad \mbox{ for all } i \in [N].
\]
\end{lemma}
\begin{proof} We consider the final filter $\w = \pi(\wall)$ and the corresponding polynomial $\w(x):=p_{\k-1}(\w)(x)$ of  degree $l \leq \k-1$ associated to it. Setting $a_{i-1}:=\w_i$ for $i\in\ul{l+1}$,  $\w(x)$ is given by
\begin{align*}
    \w(x)=a_l\cdot x^l+a_{l-1}\cdot x^{l-1}+\dots +a_1\cdot x+a_0.
\end{align*}
We first consider the case that the degree of the polynomial $\w$ satisfies $l\geq 1$.
%We consider two cases depending on the degree $l$ of the polynomial $\v(t,x)$.\\
%\textbf{1st case:} $l\geq1$\\
Then  $a_l\not= 0$ and according to  \cref{lem_PolyZerLinFak} there exist linear factors $u_1,\ldots,u_l\in\C[x]_1$ with
\begin{align}
\label{uiKL4Tk2}
    \w(x)=\sgn(a_l)\cdot \prod_{i=1}^{l}u_i(x) \quad\quad\text{and}\quad\quad \triple{u_i}_1\leq 6Tl<6T\k\quad \text{for } i\in\ul{l},
\end{align}
i.e., $\w(x)$ can be decomposed into polynomials with bounded coefficients. According to  \Cref{th_Productps}, $\w(x)$ can also be written as the product $\Pi_{i=1}^N \Tildepi(x)$ of  the polynomials $\Tildepi(x):=\Tilde{p}_{k_i,\textbf{s}}(\wi)(x) $ with $\triple{\Tildepi}_1=\Vert \wi\Vert_1$. Thus, the set of roots of $\Tildepi(x)$ is also part of the set of roots of $\w(t,x)$. So for any $j\in\ul{l}$ there is at least one $i\in\ul{N}$, so that % a multiple of 
$u_j$ divides % is also part of the linear factorization of 
$\Tildepi(x)$ and by the unique prime factorization in $\mathbb C[x]_n$, it follows that there exist subsets $M_1,\ldots,M_N$  of $\ul{l}$ and real numbers $\kappa_1,\ldots,\kappa_N$ such that $M_{i_1}\cap M_{i_2}=\emptyset$ for $i_1\neq i_2$, $\bigcup_{i\in\ul{N}} M_i=\ul{l}$ and
\begin{align}   
\label{guiyuiug}
\Tildepi(x)=\kappa_i\prod_{j\in M_i}u_j(x). 
\end{align}
Note that $M_i=\emptyset$ for $i\in\ul{l}$ with $\deg_x(\Tildepi)=0$.
It follows that 
\begin{align*} \w(x)&=\prod_{i=1}^{N}\Tildepi(x)=\prod_{i=1}^{N}\l\kappa_i\prod_{j\in M_i}u_j(x)\r=\l\prod_{i=1}^{N}\kappa_i\r\cdot\l\prod_{j\in \ul{l}}u_j(x)\r.
\end{align*}
By \eqref{uiKL4Tk2} this implies that
\begin{align*}
 \prod_{i=1}^{N}\kappa_i =\sgn(a_l).
\end{align*}
Let $\beta_i:=\Vert \wi\Vert^2_{2}=\triple{\Tildepi}_2^2$ for $i\in\ul{N}$. Then $\beta_{i+1}=\delta_i+\beta_i$ for $i\in\ul{N-1}$ and by \eqref{guiyuiug}
\begin{align*}
&\beta_1\cdot\ldots\cdot\beta_N
%\stackrel{\eqref{guiyuiug}}
=\underbrace{(\kappa_1\cdot\ldots\cdot\kappa_N)^2}_{=1}\prod_{i=1}^{N}\bigtriple{\prod_{j\in M_i}u_j}_2^{2},
\end{align*}
which is equivalent to
\begin{align*}  
   \beta_1(\beta_1+\delta_1)(\beta_1+\delta_1+\delta_2)\cdot\ldots\cdot(\beta_1+\delta_1+\ldots+\delta_{N-1})-\prod_{i=1}^{N}\bigtriple{ \prod_{j\in M_i}u_j}_2^{2} = 0.
\end{align*}
Thus, $\beta_1$ is a root of the polynomial
\begin{align*}
    z(z+\delta_1)(z+\delta_1+\delta_2)\cdot\ldots\cdot(z+\delta_1+\ldots+\delta_{N-1})-\prod_{i=1}^{N}\bigtriple{\prod_{j\in M_i}u_j}_2^{2} .
\end{align*}
The leading coefficient of this polynomial is one. The absolute value of the coefficient of $z^{N-n}$ for $n\in\ul{N-1}$ satisfies
\begin{align*}
& \Big\vert\sum_{\substack{j_1<\ldots<j_n\\ j_k\in\ul{N}}}\l\sum_{i=1}^{j_1-1}\delta_i\r\cdot\ldots\cdot\l\sum_{i=1}^{j_n- 1}\delta_i\r\Big\vert\\
&\leq \sum_{\substack{j_1<\ldots<j_n\\ j_k\in\ul{N}}}(\max_{i\in\ul{N-1}} \vert\delta_i\vert)^n(j_1-1)\cdot\ldots\cdot (j_n-1).%\\
%\leq (\max_{i\in\ul{N-1}} \vert\delta_i\vert)^n(N-1)!\binom{N-1}{n}
%\leq \max\{\max_{i\in\ul{N-1}} \vert\delta_i\vert,1\}^N((N-1)!)^2.
\end{align*}
The constant term can be bounded as 
\begin{align*}
  \prod_{i=1}^{N}\bigtriple{\prod_{j\in M_i}u_j}_2^{2}\leq \prod_{i=1}^{N}\bigtriple{\prod_{j\in M_i}u_j}_1^{2}\leq \prod_{i=1}^{l}\bigtriple{u_i}_1^{2} <  (6T\k)^{2\k},
\end{align*}
where we used \eqref{uiKL4Tk2} and the fact that the $1$-norm is submultiplicative according to \Cref{lem_NormSubmulti}. 
 It follows from  \Cref{lem_rootsBound}
that $\beta_1$ is bounded by a constant that only depends on $\k,T,\delta_1,\ldots,\delta_N$ and since $\beta_{i+1}=\delta_i+\beta_i$ for $i\in \ul{N-1}$, it follows that all $\beta_i$ are bounded by a  constant $\tau=\tau(\k,T,\delta_1,\ldots,\delta_N)$ that only depends on $\k,T,\delta_1,\ldots,\delta_N$.

%STILL TO BE EDITED:

Let us now consider the case that $l=0$, i.e., that the polynomial $v$ is constant.
%\textbf{2nd case:} $l= 0.$\\
If the constant is zero, i.e, $\w(x)\equiv 0$ then at least one of the polynomials $\Tilde{p}_{j}$ in the factorization of $\w$ (see Proposition~\ref{th_Productps}) 
is zero, $\Tilde{p}_{j}(x)\equiv 0$.
%there exists a $j\in\ul{N}$ with $\Tilde{p}_{j}(x)\equiv 0$. 
Defining $\beta_i = \triple{\Tilde{p}_{i}}_2^2$ as above, it follows that  $\delta_j=\beta_{j+1}-\beta_{j}=\beta_{j+1}$ and, as a consequence 
\begin{align*}
    \beta_i\leq (N-1)\max_{k\in\ul{N-1}}\vert\delta_k\vert\ \quad \mbox{ for all } i \in [N]. %leq \tau.
\end{align*}
If $\w$ is a constant polynomial $\w \not\equiv 0$ it follows that all $\Tildepi(x)$ are constant as well. This allows to apply the modulus in each factor individually,
\begin{align*}
    \triple{\w}_2^2= \bigtriple{\prod_{i=1}^N \Tildepi}_2^2= \prod_{i=1}^N \triple{\Tildepi}_2^2= \prod_{i=1}^N\beta_i .
\end{align*}
Together with $\beta_{i+1}=\delta_i+\beta_i$ this yields
\begin{align*}
    \beta_1(\beta_1+\delta_1)(\beta_1+\delta_1+\delta_2)\cdot\ldots\cdot(\beta_1+\delta_1+\ldots+\delta_{N-1})-\triple{ \w}_2^2=0.
\intertext{Thus $\beta_1$ is a root of the polynomial}
z(z+\delta_1)(z+\delta_1+\delta_2)\cdot\ldots\cdot(z+\delta_1+\ldots+\delta_{N-1})-
\triple{ \w}_2^2 .
\end{align*}
By an analogous argument as in the case $l \geq 1$ together with \Cref{lem_rootsBound}, 
we obtain again that all 
$\beta_i$ are bounded by a  constant $\tau=\tau(\k,T,\delta_1,\ldots,\delta_N)$ that only depends on  $\k,T,\delta_1,\ldots,\delta_N$.
\end{proof}
\begin{proof}[Proof of Theorem~\ref{th_Main}]
The first point follows directly from standard existence results (Picard-Lindelöf Theorem) for ordinary differential equations. %  \Cref{Koenigsberger}, point (1), which assumes that $\nabla\LosskurzN$ is locally Lipschitz continuous, which
 Note that local Lipschitz continuity follows from the continuous differentiability of $\nabla\LosskurzN$. 
 %This confirms the first point. 

For the second point, 
%we have to show that the solution of the initial value problem \ref{gradient flow}  converges to a critical point of $\Losskurz$ as $t\to\infty$. 
%For  this in turn, 
we first show that the ODE system \eqref{gradient flow} has a solution that is defined for all $t\in\left[0,\infty \right)$, i.e. $\left[0,\infty \right)\subseteq I$. 
%But this follows once 
%we have shown that its maximal integral curve stays bounded on $\left[0,\infty \right)\cap I$ since 
It is well known that if the maximal integral curve of a (locally Lipschitz) initial 
value problem starting in $t=0$ is not defined on $\left[ 0,\infty  \right)$, then the maximal integral curve leaves any compact set, in particular, the integral curve is then unbounded  on $\left[ 0,\infty \right)\cap I$.
Combining this observation with \Cref{LGIbeschrKonvKritPunkt}, we see that the claim of point 2)  follows once we show that the maximal integral curve stays bounded on $\left[0,\infty \right)\cap I$.

For the maximal integral curve $\wall(t)=\left( w^{(1)}(t)^\top,\ldots,w^{(N)} (t)^{\top}\right)^\top$ we define the associated final filter by $v(t):=\pi(\wall(t))$ with filter width \k and stride \s. Let $U=[0,\infty)\cap I$. We first show boundedness of $v(t)$ on $U$.

Note that $\Losskurz\l\wall\l{t}\r\r$ is decreasing in $t$ for $t\in  [0,\infty)\cap I$ since
\begin{align*}
    \frac{d}{dt}\Losskurz\l\wall\l t\r\r &=\l\nabla\Losskurz\l\wall\l t\r\r\r^\top\cdot \frac{d}{dt}\wall(t)
    %\stackrel{\eqref{gradient flow}}{=}
    = -\l\frac{d}{dt}\wall(t)\r^\top \cdot\frac{d}{dt}\wall(t) = -\Vert \frac{d}{dt}\wall(t)\Vert_2^2 \\ & \leq 0.
\end{align*}
%So $\Losskurz$ is indeed decreasing. 
Together with  condition \eqref{eq_EndfilterBoundedbyL} and the monotonicity of $g$ it follows for $t\in U$ that
\begin{align*}
    \Vert\w(t)\Vert_1\leq g\l\Losskurz\l\wall(t)\r\r\leq g\l\Losskurz\l\wall(0)\r\r =: T.
\end{align*}
In particular, $\w(t)$ is bounded on $U$. % so that $\w(t)$ is defined for all $t \in [0,\infty)$.
%This means in particular that $T=\displaystyle\sup_{t\in U}\{\Vert \w(t)\Vert_1,1\}<\infty$, i.e., $\w(t)$ is bounded. % This confirms the first claim. %$\quad\qed$
%\endgroup
%EDITED UNTIL HERE!
It follows from Lemma~\ref{lem:bound-individual} that there exists a constant $\tau = \tau(T, k_v, \delta_1,\hdots,\delta_{N-1})$ such that for all $t \in U$ and all $i \in [N]$
\[
\|w^{i}(t)\|_2^2 \leq \tau.
\]
In particular, $\wall(t)$ is bounded on $U$. Hence, $\wall(t)$ is defined for all $t \in [0,\infty)$ and as argued already above, it follows from Theorem~\ref{LGIbeschrKonvKritPunkt} that $\wall(t)$ converges to a critical point of $\Losskurz$.
\end{proof}

Next we provide some criteria that imply condition \eqref{eq_EndfilterBoundedbyL} of \Cref{th_Main}. We will use the entry-wise $\ell_1$-norm of a matrix Let $A=(a_1,\dots,a_m)\in \R^{l\times m}$ with columns $a_1,\dots,a_m$, defined as
$\Vert A\Vert_1:= \sum_{i=1}^l\sum_{j=1}^m \vert a_{ij}\vert=\sum_{j=1}^m \Vert a_{j}\Vert_1$. It is straightforward to verify that the matrix $\ell_1$-norm is submultiplicative, i.e., for matrices 
$A\in\R^{l\times m}$ and $B\in\R^{m\times n}$ it holds
\begin{align}\label{l1-submultiplicative}
        \Vert AB\Vert_1\leq \Vert A\Vert_1\Vert B\Vert_1.
\end{align}

%To make it easier to check condition 2) of \Cref{th_Main}, sufficient conditions are formulated below. 
Below we restrict ourselves to training data for which the input matrix $X$ has full rank.% $d_0$. 

\begin{lemma}
\label{Cor_EasyCond2}
 Let $X \in \R^{d_x \times m}, Y \in \R^{d_y \times m}$. Assume that $XX^\top$ has full rank and that there exists a monotonically increasing function $h: \R \to \R_{\geq 0}$ such that
 \begin{align}
 \label{eq_WBoundedbyEll}
\Vert W X - Y \Vert_1 \leq h\Bigg(\sum_{i=1}^{m}\ell(Wx_i,y_i)\Bigg) \quad  \mbox{ for all } W \in \R^{d_N \times d_0}.
\end{align}
Then \eqref{eq_EndfilterBoundedbyL} in \Cref{th_Main} holds.
% Then it is sufficient for the condition on $\Losskurz$ in point (2) of \Cref{th_Main} that there exists a function $g: \R \to \R_{\geq 0}$ which  increases monotonically  on $\Losskurz(\wall(I))$ and for which
\end{lemma}
%The proof generalizes the proof for the boundedness of the product matrix of the individual layer matrices  in linear neural networks 
%with $\ell_2$-loss (see \Cref{subsubsecBspLosses} below for the definition), as discussed in \cite[Theorem 3.2]{BRTW}.
%matrix by the functional $\Losskurz$ for linear neural networks with $\ell_2$-loss (see \Cref{subsubsecBspLosses} below for the definition), as shown in \cite[Theorem 5]{BRTW}.
\begin{proof} We proceed similarly as in the proof of \cite[Theorem 3.2]{BRTW}.
For some $\wall$ %let $v = \pi(\wall)$ and
let $\W:=\Pi_{\archi}(\wall)$. The submultiplicity \eqref{l1-submultiplicative} of the matrix $\ell_1$-norm gives
    \begin{align*}
        \Vert \pi(\wall)\Vert_1 &=\frac{1}{d_N}\Vert V\Vert_1= \frac{1}{d_N}\Vert VXX^\top\l XX^\top\r^{-1}\Vert_1\\
        &\leq \frac{1}{d_N}\Vert V X\Vert_1\Vert X^\top\l XX^\top\r^{-1}\Vert_1\\
         & = \frac{1}{d_N}\Vert VX-Y+Y\Vert_1\Vert X^\top\l XX^\top\r^{-1}\Vert_1\\
        & \leq \frac{1}{d_N}\l\Vert VX-Y\Vert_1+\Vert Y\Vert_1\r\Vert X^\top\l XX^\top\r^{-1}\Vert_1\\
        &\leq \frac{1}{d_N}\l h\l\sum_{i=1}^{m}\ell(Vx_i,y_i)\r+\Vert Y\Vert_1\r\Vert X^\top\l XX^\top\r^{-1}\Vert_1\\\
        & = \frac{1}{d_N}\l h\l\Losskurz(\wall)\r+\Vert Y\Vert_1\r\Vert X^\top\l XX^\top\r^{-1}\Vert_1 =: g(\Losskurz(\wall)).
    \end{align*}
  Since $h$ is non-decreasing, also $g$ is non-decreasing which establishes condition \eqref{eq_EndfilterBoundedbyL} of Theorem~\ref{th_Main}.
   % where $\frac{1}{d_N}\l g\l x\r+\Vert Y\Vert_1\r\Vert X^\top\l XX^\top\r^{-1}\Vert_1$ is monotonically increasing to $\Losskurz(\wall(I))$ due to the condition on $g$.
\end{proof}

The next corollary slightly simplifies condition \eqref{eq_WBoundedbyEll}.
%on $\Losskurz$. %\medskip

\begin{cor}
\label{Cor_EasierCond2}
Let $X \in \R^{d_x \times m}$. Assume that $XX^\top$ has full rank and that, for all $x \in \R^{d_x}, y \in \R^{d_y}$, 
\begin{align}\label{cond-easy}
     \Vert W x-y\Vert_1\leq \ell(Wx,y)+c \quad \mbox{ for all }W \in \R^{d_N \times d_0}.
 \end{align}
Then \eqref{eq_EndfilterBoundedbyL} in \Cref{th_Main} holds.
\end{cor}
\begin{proof}
    The statement follows directly from \Cref{Cor_EasyCond2} because
    %the prerequisite applies:
    \begin{align*}
        \Vert W X - Y \Vert_1 =\sum_{i=1}^{m} &\Vert W x_i-y_i\Vert_1 \leq \sum_{i=1}^{m} \l\ell(Wx_i,y_i)+c\r = \sum_{i=1}^{m}\ell(Wx_i,y_i)+c\cdot m.
        %\\
        %&\Longrightarrow \Vert W X - Y \Vert_1 \leq g\l \sum_{i=1}^{m}\ell(Wx_i,y_i)\r ,
    \end{align*}
    Hence, we have $\Vert W X - Y \Vert_1 \leq h\l \sum_{i=1}^{m}\ell(Wx_i,y_i)\r$ with the monotonically increasing function $h(t)=t+cm$.
\end{proof}
\bigskip

\subsection{Loss functions with convergent gradient flow}
\label{subsubsecBspLosses}

In this subsection, we discuss some examples of loss functionals  that fulfill the conditions of Theorem~\ref{th_Main}. For this purpose, the architecture \archi\  and the training data $X,Y$ are fixed. We use the same shorthand notation as before. %\medskip

Let us consider the $\ell_p$ loss given by $\ell_p(z,y):=\frac{1}{p}\Vert z-y\Vert_p^p = \frac{1}{p}\sum_{j=1}^{d_y} \vert z_j-y_j\vert^p$ for $p\geq 1$ and $z,y \in \R^{d_y}$. We also use the matrix-$p$-norm defined for $A=(a_1,\dots,a_m)\in \R^{l\times m}$ as $\Vert A\Vert_p:=\l \sum_{i=1}^m \Vert a_i\Vert_p^p\r^\frac{1}{p}$.
%the dimension of $x$ and $y$. 
The loss functional for $\ell_p$ is given by
\begin{align*}
   \mathcal{L}_p(\wall):=\sum_{i=1}^m\ell_p(\Pi(\wall)\cdot x_i,y_i) = \frac{1}{p}\Vert \Pi(\wall)\cdot X-Y\Vert_p^p.
\end{align*}
For even $p$, $\mathcal{L}_p(\wall)$ is analytic (note that the vectors $\Pi_{\archi}(\wall)\cdot x_i$ have entries that are polynomial in the entries of  $\wall$, and that since $p$ is even the absolute values in the  definition of the $\ell_p$ loss can be omitted).  Thus for $p$ even, $\mathcal{L}_p(\wall)$    fulfills the \LGU{} by \L ojasiewz's Theorem.
For other values of $p\geq 1$ however, the absolute values in  definition of the $\ell_p$ loss prevents analyticity so that
it is not clear whether or under which conditions $\mathcal{L}_p$ satisfies the \LGU{} for non-even $p$.

The $\ell_1$ loss and the $\ell_2$ loss are loss functions that are frequently used in machine learning practice. The advantage of the $\ell_2$-loss is strict convexity. On the other hand, the $\ell_2$ loss function is sensitive to outliers. The $\ell_1$-loss is robust to outliers, but not differentiable. The best-known approach to combine the advantages of the two loss functions is the Huber loss, which, however, is not continuously differentiable. The alternative approaches   pseudo-Huber-loss, generalized-Huber-loss and $\log\text{-}\cosh$-loss are presented next.  These three loss functions behave  similarly to the $\ell_1$-loss for $\vert x\vert\to \infty$ and approach the $\ell_2$-loss for $\vert x\vert\to 0$, see \Cref{fig:enter-label} for an illustration. % \cite{GHL}.

\begin{figure}[h!]
    \centering
    \includegraphics[width=0.75\linewidth]{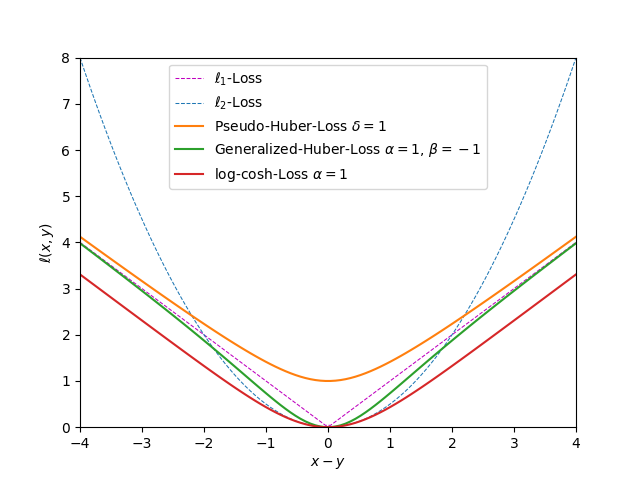}
    \caption{The figure shows the various loss functions for differences of $x-y\in [-4,4]$.}
    \label{fig:enter-label}
\end{figure}

The \textit{pseudo-Huber-loss} \cite{GHL} with respect to $\delta>0$ is defined by
\begin{align*}
    \ell_{PH}(x,y):=\delta\sum_{i=1}^n\sqrt{1+\frac{\l x_i-y_i\r^2}{\delta^2}} \qquad \text{for } x,y\in\R^{n}.
\end{align*}
For $\alpha>0$ and $\beta+2>0$ the \textit{generalized-Huber-loss }\cite{GHL}
is defined as
\begin{align*}
    \ell_{GH}(x,y):=\frac{1}{\alpha}\sum_{i=1}^n\log\l \e^{ \alpha\l x-y\r_i}+\e^{ -\alpha\l x-y\r_i}+\beta\r \qquad \text{for } x,y\in\R^{n}.
\end{align*}
The \textit{$\log$-$\cosh$-loss} is defined for $\alpha>0$ and $x,y\in\R^{n}$ as  
\begin{align*}
    \ell_{\log\text{-}\cosh}(x,y):=\frac{1}{\alpha}\sum_{i=1}^n\log\cosh\l\alpha\l x-y\r_i\r \qquad \text{for } x,y\in\R^{n}.
    %= \frac{1}{\alpha}\sum_{i=1}^n\log\l\frac{ \e^{\alpha \l x-y\r_i}+\e^{-\alpha \l x-y\r_i}}{2}\r.
\end{align*}

The next theorem states convergence of the gradient flow for learning linear convolutional networks to a critical point for the loss functions defined above.
 \begin{theorem}
    \label{th_Main2}
        Assume that $XX^\top$ has full rank and let $\Losskurz(\wall):=\sum_{i=1}^m\ell(\Pi(\wall)\cdot x_i,y_i)$ be the loss functional with respect to one of the following loss functions:
        \begin{itemize}
             \item $\ell_p$-loss with $p\in\N$ even;
            \item pseudo-Huber-loss $\ell_{PH}$ with $\delta>0$;
            \item generalized-Huber-loss $\ell_{GH}$ with $\alpha>0$ and $\beta>-2$;
            \item $\log$-$\cosh$-loss $\ell_{\log\text{-}\cosh}(x,y)$ with $\alpha>0$.
        \end{itemize}
        Then the gradient flow $\wall(t)$   with respect to the gradient system \eqref{gradient flow}  is defined and  bounded for all $t\geq 0$ and $\wall(t)$ converges for $t\to\infty$ to a critical point of $\Losskurz.$
    \end{theorem}
    \begin{proof}
     We check the conditions of \Cref{th_Main}. All specified loss functions are themselves analytic as compositions of analytic functions. The function that maps $\wall$ for fixed $x\in\R^{d_0}$ to $\Pi(\wall)\cdot x$ represents componentwise a multivariate polynomial that depends on the entries of \wall and is therefore also analytic. Since sums of analytic functions are also analytic, \Losskurz\ is also analytic. Thus \Losskurz\ fulfills \LGU{} and is twice continuously differentiable. 
     
     It remains to check point (2) of \Cref{th_Main} for all listed loss functions. %, for which we use Corollary~\ref{Cor_EasierCond2}. 
     We start with the $\ell_p$-loss. 
     For $x\in\R^{n}$ and $p\geq 1$ Hölder's inequality
     implies that $\Vert x\Vert_1 \leq n^{1-\frac{1}{p}}\Vert x\Vert_p$.
     %we have  %and using $\frac{1}{p}+\l 1-\frac{1}{p}\r=1$ 
    %for $p>1$:  $$\Vert x\Vert_1=\sum_{i=1}^{n}\l\vert x_i\vert\cdot 1\r\leq \l\sum_{i=1}^{n}\vert x_i\vert^{p}\r^{\frac{1}{p}} \cdot \l\sum_{i=1}^{n} 1^{\l 1-\frac{1}{p}\r^{-1}}\r^{1-\frac{1}{p}}= n^{1-\frac{1}{p}}\Vert x\Vert_p.$$ %Together with $\Vert A\Vert_p=\Vert \widetilde{a}\Vert_p$ 
    It  follows that 
    \begin{align*}
        \Vert W X-Y\Vert_1
        &\leq (d_N\cdot m)^{1-\frac{1}{p}}\cdot\Vert W X-Y\Vert_p\\\
         &\leq d_N\cdot m\cdot p^\frac{1}{p}\cdot\l\frac{1}{p}\Vert W X-Y\Vert_p^p\r^\frac{1}{p}
         = h\l\sum_{i=1}^{m}\ell_p(Wx_i,y_i)\r,
    \end{align*}
    where $h(x):=d_N\cdot m\cdot p^\frac{1}{p}\cdot x^\frac{1}{p}$ is increasing for $x\geq0 $. It follows from Lemma~\ref{Cor_EasyCond2} that condition 2) of Theorem~\ref{th_Main} holds.

    For the remaining loss functions, we will use Corollary~\ref{Cor_EasierCond2}. 
    For the pseudo-Huber-loss and
    $W\in\R^{d_N\times d_0}$, we have
    \begin{align*}
        \Vert W x-y\Vert_1
        =\sum_{j=1}^{d_N} \vert \l Wx-y\r_j\vert
        =\delta \sum_{j=1}^{d_N} \sqrt{ \frac{\l\l Wx\r_j-y_j\r^2}{\delta^2}}
        \leq \ell_{PH}(Wx,y)
    \end{align*}
    for all $x\in \R^{d_0}$ and $y\in \R^{d_N}$. 

For the generalized-Huber-loss
and $W\in\R^{d_N\times d_0}$, we have 
    \begin{align*}
        \Vert W x-y\Vert_1
        &=\frac{1}{\alpha}\sum_{j=1}^{d_N} \alpha\vert \l Wx-y\r_j\vert
        =\frac{1}{\alpha} \sum_{j=1}^{d_N} \log\e^{\alpha\vert \l Wx-y\r_j\vert}\\\
       & \leq\frac{1}{\alpha} \sum_{j=1}^{d_N} \log\l \e^{\alpha \l Wx-y\r_j}+\e^{-\alpha \l Wx-y\r_j}\r .
    \end{align*}
    Since the logarithm is monotonically increasing, it follows directly for $\beta\geq0$ that 
    \begin{align}
    \label{glGH}
        \Vert W x-y\Vert_1\leq \frac{1}{\alpha} \sum_{j=1}^{d_N} \log\l \e^{\alpha \l Wx-y\r_j}+\e^{-\alpha \l Wx-y\r_j}\r
        \leq \ell_{GH}(Wx,y)
    \end{align}
    for all $x\in \R^{d_0}$ and $y\in \R^{d_N}$.
    For $\beta\in(-2,0)$ we obtain
    %use the property of the logarithm that we can separate products in the argument as a sum:
    \begin{align*}
        \Vert W x-y\Vert_1&\leq \frac{1}{\alpha} \sum_{j=1}^{d_N} \log\l \e^{\alpha \l Wx-y\r_j}+\e^{-\alpha \l Wx-y\r_j} +\beta -\beta\r\\
            &=\frac{1}{\alpha} \sum_{j=1}^{d_N} \log\left(\Big(\e^{\alpha \l Wx-y\r_j}+\e^{-\alpha \l Wx-y\r_j} +\beta\Big)\right. \\
            &\qquad\qquad \quad
            \cdot\left.\left( 1 +\frac{-\beta}{\e^{\alpha \l Wx-y\r_j}+\e^{-\alpha \l Wx-y\r_j} +\beta}\right)\right)\\
            &=\frac{1}{\alpha} \sum_{j=1}^{d_N} \left( \log\Big(\e^{\alpha \l Wx-y\r_j}+\e^{-\alpha \l Wx-y\r_j} +\beta\big) \right. \\
            & \qquad\qquad \left.+\log\l 1 +\frac{-\beta}{\e^{\alpha \l Wx-y\r_j}+\e^{-\alpha \l Wx-y\r_j} +\beta}\r\right)\\
            &\stackrel{\beta<0}{\leq} \frac{1}{\alpha} \sum_{j=1}^{d_N} \l \log\Big(\e^{\alpha \l Wx-y\r_j}+\e^{-\alpha \l Wx-y\r_j} +\beta\Big)+\log\l 1 +\frac{-\beta}{2 +\beta}\r\r\\
            &= \frac{1}{\alpha} \sum_{j=1}^{d_N} \log\l\e^{\alpha \l Wx-y\r_j}+\e^{-\alpha \l Wx- y\r_j} +\beta\r+\frac{d_N}{\alpha} \cdot\log\l \frac{2}{2 +\beta}\r\\
            &= \ell_{GH}(Wx,y) + \frac{d_N}{\alpha} \cdot\log\l \frac{2}{2 +\beta}\r .
    \end{align*}
    Note that for $\beta=0$ the generalized-Huber-loss is related to the $\log$-$\cosh$-loss via $\ell_{GH}(Wx,y) = \ell_{\log\text{-}\cosh}(Wx,y)+\frac{d_N\log(2)}{\alpha}$. Therefore by the above, for $x\in\R^{d_0}$ and $y\in\R^{d_N}$ it holds
\begin{align*}
     \Vert W x-y\Vert_1 
     \leq \ell_{\log\text{-}\cosh}(Wx,y)+\frac{d_N\log(2)}{\alpha}.
\end{align*}
    Hence, we established \eqref{cond-easy} for the pseudo-Huber-loss, generalized-Huber-loss and the $\log$-$\cosh$-loss and the statement follows from \Cref{th_Main}.
    \end{proof}

    Note that the important question remains as to whether the critical points to which the gradient flow converges are also  minimizers of the functional and, if so, whether they are global. One result in this direction is Proposition 5.1 in \cite{Geo} which states that for architectures where $s=(1,\ldots , 1)$ and $\ell$ is convex, a critical point $\wall$ is a global minimum if the polynomials $\Tilde{p}_{k_1,\textbf{s}}(\wI{1})(x),\ldots ,\Tilde{p}_{k_N,\textbf{s}}(\wI{N})(x) $, see \eqref{eq_SpecialPolynomConstruction}, have no common roots with respect to $x$. Further investigations in this direction are postponed to future contributions.
 \printbibliography 
\section*{Appendix}
\addcontentsline{toc}{section}{\numberline{}{Appendix}} 
\label{Appendix}

\subsection{Proof of Proposition~\ref{th_CompConv}}
%\begin{proof}
We proceed by induction on $N$. For an architecture 
\[
(\textbf{d},\textbf{k},\textbf{s})=((d_0,\ldots , d_N),(k_1,\ldots , k_N),(s_1,\ldots , s_N)),
\]
let the mappings $\alpha_i: \R^{d_{i-1}}\to\R^{d_i}$ be convolutions with filters $\wi\in\R^{k_i}$ for $i\in\ul{N}$.
For $N=1$ the statement is trivial.
Assume that the statement is satisfied for $N-1\in\N$. Then
\begin{align*}
    (\alpha_{N}\circ\alpha_{N-1}\circ\ldots\circ\alpha_{1})(x)= (\alpha_{N}\circ\tilde{\alpha}_{N-1})(x),
\end{align*}
where according to the induction assumption $\tilde{\alpha}_{N-1}$ is a convolution with filter width $\tilde{k}:=k_1+\sum_{i=2}^{N-1}(k_i-1)\prod_{m=1}^{i-1}s_m$ and stride $\tilde{s}:=\prod_{i=1}^{N-1} s_i$. Define $\tilde{\w}\in\R^{\tilde{k}}$ as the filter associated with $\tilde{\alpha}_{N-1}$. 
Then we have
\begin{align*}
 & (\alpha_{N}\circ\tilde{\alpha}_{N-1})(x)_j =\sum_{n=1}^{k_{N}}w^{(N)}_{n}\cdot \tilde{\alpha}_{N- 1}(x)_{(j-1)\cdot s_{N}+n}  \\
&=\sum_{n=1}^{k_{N}}w^{(N)}_{n}\cdot\sum_{l=1}^{\tilde{k}}\tilde{\w}_{l} \cdot x_{((j-1)\cdot s_{N}+n-1)\cdot \tilde{s}+l}\\
 &=\sum_{n=1}^{k_{N}}\sum_{l=1}^{\tilde{k}}w^{(N)}_{n}\cdot \tilde{\w}_{l} \cdot x_{(j-1)\cdot s_{N}\cdot \tilde{s}+(n-1)\cdot \tilde{s}+l}
  = \sum_{m=1}^{(k_{N}-1)\cdot \tilde{s}+\tilde{k}}\w_{m}\cdot x_{(j-1)\cdot s_{N}\cdot \tilde{s}+m}.
\end{align*}
Thus, $\alpha_{N}\circ\tilde{\alpha}_{N-1}$ is a convolution with filter width $\k=(k_{N}-1)\cdot \tilde{s}+\tilde{k}=k_1+\sum_{i=2}^{N}(k_i-1)\prod_{m=1}^{i-1}s_m\; $, stride $\s=s_{N}\cdot \tilde{s}=\prod_{i=1}^{N} s_i\;$ and
\begin{align}
\label{eq_endfilter}
    \w_{m} = \sum_{\substack{n\in \ul{k_N} ,\, l\in\ul{\tilde{k}}\\\ (n-1)\cdot \tilde{s}+l=m}} w^{(N)}_{n}\cdot \tilde{\w}_{l}.
\end{align}
This completes the proof.
\end{document}